\documentclass[12pt]{amsart}

\usepackage{amssymb}
\usepackage{ifthen}
\usepackage{xargs}
\usepackage{xspace}

\makeatletter
\@namedef{subjclassname@2010}{%
  \textup{2010} Mathematics Subject Classification}
\makeatother

\newtheorem{lemma}{Lemma}[section]
\newtheorem{proposition}[lemma]{Proposition}
\newtheorem{theorem}[lemma]{Theorem}
\newtheorem{introtheorem}{Theorem}

\theoremstyle{definition}

\newcommand{\definedterm}[1]{\emph{#1}}

\newcommand{\Bairespace}[1][]{
  \ifthenelse{\equal{#1}{}}{\functions{\N}{\N}}{\functions{#1}{\N}}
}
\newcommand{\Bairetree}[1][]{
  \ifthenelse{\equal{#1}{}}{\functions{<\N}{\N}}{\functions{#1}{\N}}
}
\newcommand{\ball}[3]{\calB_{#1}(#2, #3)}
\newcommand{\calB}{\mathcal{B}}
\newcommand{\calN}{\mathcal{N}}
\newcommand{\Cantorspace}[1][]{
  \ifthenelse{\equal{#1}{}}{\functions{\N}{2}}{\functions{#1}{2}}
}
\newcommand{\Cantortree}[1][]{
  \ifthenelse{\equal{#1}{}}{\functions{<\N}{2}}{\functions{#1}{2}}
}
\newcommand{\closure}[2][]{
  \ifthenelse{\equal{#1}{}}{\overline{#2}}{\overline{#2}^{\thinspace #1}}
}
\newcommand{\composition}{\circ}
\newcommandx{\concatenation}[2][1 = undefined, 2 = undefined]{
  \ifthenelse{\equal{#1}{undefined}}{{}\smallfrown}{
    \ifthenelse{\equal{#2}{undefined}}{\bigoplus #1}{\bigoplus_{#1} #2}
  }
}
\newcommandx{\Deltaclass}[2][1=,2=]{
  \ifthenelse{\equal{#2}{}}{\mathbf{\Delta}_{#1}}{\mathbf{\Delta}^{#1}_{#2}}
}
\newcommand{\diameter}[2]{\mathrm{diam}_{#1}(#2)}
\newcommand{\emptysequence}{\emptyset}
\newcommand{\extendedBairespace}{\Bairespace[\le \N]}
\newcommand{\extendedby}{\sqsubseteq}
\newcommand{\extendedextensions}[1]{\extensions{#1}^*}
\newcommand{\extensions}[1]{\calN_{#1}}
\newcommand{\fone}{f_1}
\newcommand{\from}{\colon}
\newcommand{\Fsigma}{$F_\sigma$\xspace}
\newcommandx{\functions}[3][3 =]{
  \ifthenelse{\equal{#3}{}}{#2^{#1}}{#2^{#1}_{#3}}
}
\newcommand{\fzero}{f_0}
\newcommand{\Gdelta}{$G_\delta$\xspace}
\newcommand{\goesto}{\rightarrow}
\newcommand{\graph}[1]{\mathrm{graph}(#1)}
\newcommand{\Gzero}{\mathbb{G}_0}
\newcommand{\Gzerolittle}{\mathbb{G}_0^{\Deltaclass[0][2]}}
\newcommand{\GzeroN}[1][]{
  \ifthenelse{\equal{#1}{}}{\mathbb{G}_0^\N}{\mathbb{G}_{0, #1}^{\N}}
}
\newcommand{\horizontalsection}[2]{#1^{#2}}
\newcommand{\image}[2]{#1(#2)}
\newcommand{\incompatible}{\perp}
\newcommandx{\intersection}[2][1 =, 2 =]{
  \ifthenelse{\equal{#1}{}}{\cap}{
    \ifthenelse{\equal{#2}{}}{\bigcap #1}{{\bigcap_{#1} #2}}
  }
}
\newcommand{\length}[1]{|#1|}
\newcommand{\mathand}{\text{ and }}
\newcommand{\N}{\mathbb{N}}
\newcommand{\pair}[2]{(#1, #2)}
\newcommandx{\Piclass}[2][1=,2=]{
  \ifthenelse{\equal{#2}{}}{\mathbf{\Pi}_{#1}}{\mathbf{\Pi}^{#1}_{#2}}
}
\newcommand{\preimage}[2]{#1^{-1}(#2)}
\newcommandx{\projection}[2][1 =, 2 =]{
  \ifthenelse{\equal{#1}{}}{\mathrm{proj}}{
    \ifthenelse{\equal{#2}{}}{\projection_{#1}}{
      \image{\projection[#1]}{#2}
    }
  }
}
\renewcommand{\restriction}[2]{#1 \upharpoonright #2}
\newcommandx{\sequence}[2][2 = undefined]{
  \ifthenelse{\equal{#2}{undefined}}{(#1)}{
    (#1)_{#2}
  }
}
\newcommandx{\set}[2][2 = undefined]{
  \ifthenelse{\equal{#2}{undefined}}{\{ #1 \}}{
    \{ #1 \suchthat #2 \}
  }
}
\newcommandx{\Sigmaclass}[2][1=,2=]{
  \ifthenelse{\equal{#2}{}}{\mathbf{\Sigma}_{#1}}{\mathbf{\Sigma}^{#1}_{#2}}
}
\newcommand{\strictlyextendedby}{\sqsubset}
\newcommand{\strongextendedBairespace}{\extendedBairespace_{**}}
\newcommand{\suchthat}{\mid}
\newcommand{\R}{\mathbb{R}}

\newcommandx{\union}[2][1 =, 2 =]{
  \ifthenelse{\equal{#1}{}}{\cup}{
    \ifthenelse{\equal{#2}{}}{\bigcup #1}{{\bigcup_{#1} #2}}
  }
}

\newcommand{\weakextendedBairespace}{\extendedBairespace_*}

\newcommand{\Baire}{Baire\xspace}
\newcommand{\Borel}{Bor\-el\xspace}
\newcommand{\Hurewicz}{Hur\-e\-wicz\xspace}
\newcommand{\Jayne}{Jayne\xspace}
\newcommand{\Lecomte}{Lec\-omte\xspace}

\newcommand{\Polish}{Po\-lish\xspace}
\newcommand{\Rogers}{Rog\-ers\xspace}
\newcommand{\SaintRaymond}{Saint Ray\-mond\xspace}
\newcommand{\Solecki}{Sol\-eck\-i\xspace}
\newcommand{\Zeleny}{Zel\-e\-ny\xspace}

\begin{document}


\baselineskip=17pt

\title[Sigma-continuity]{Sigma-continuity with closed witnesses}

\author[R. Carroy]{Rapha\"{e}l Carroy}
\address{
  Rapha\"{e}l Carroy \\
  Kurt G\"{o}del Research Center for Mathematical Logic \\
  Universit\"{a}t Wien \\
  W\"{a}hringer Stra{\ss}e 25 \\
  1090 Wien \\
  Austria
}
\email{raphael.carroy@univie.ac.at}
\urladdr{
  http://www.logique.jussieu.fr/~carroy/indexeng.html
}

\author[B.D. Miller]{Benjamin D. Miller}
\address{
  Benjamin D. Miller \\
  Kurt G\"{o}del Research Center for Mathematical Logic \\
  Universit\"{a}t Wien \\
  W\"{a}hringer Stra{\ss}e 25 \\
  1090 Wien \\
  Austria
 }
\email{benjamin.miller@univie.ac.at}
\urladdr{
  http://www.logic.univie.ac.at/benjamin.miller
}

\date{}

\begin{abstract}
  We use variants of the $\Gzero$ dichotomy to establish a
  refinement of \Solecki's basis theorem for the family of
  \Baire-class one functions which are not $\sigma$-continuous
  with closed witnesses.
\end{abstract}



\maketitle

\section*{Introduction}

A subset of a topological space is \definedterm{\Fsigma} if it is a union
of countably-many closed sets, \definedterm{\Borel} if it is in the
$\sigma$-algebra generated by the closed sets, and \definedterm
{analytic} if it is a continuous image of a closed subset of $\Bairespace$.

A function between topological spaces is \definedterm
{$\sigma$-continuous with closed witnesses} if its domain is a union of
countably-many closed sets on which it is continuous, \definedterm{\Baire
class one} if preimages of open sets are \Fsigma, \definedterm{strongly
$\sigma$-closed-to-one} if its domain is a union of countably-many analytic
sets intersecting the preimage of each singleton in a closed set,
\definedterm{\Fsigma-to-one} if the preimage of each singleton is
\Fsigma, and \definedterm{\Borel} if preimages of open sets are \Borel.

A \definedterm{topological embedding} of a topological space $X$ into
a topological space $Y$ is a function $\pi \from X \to Y$ which is a
homeomorphism onto its image, where the latter is endowed
with the subspace topology. A \definedterm{topological embedding}
of a set $A \subseteq X$ into a set $B \subseteq Y$ is a topological
embedding $\pi$ of $X$ into $Y$ such that $A = \preimage{\pi}{B}$.
A \definedterm{topological embedding} of a function
$f \from X \to Y$ into a function $f' \from X' \to Y'$ is a pair $\pair{\pi_X}
{\pi_Y}$, consisting of topological embeddings $\pi_X$ of $X$ into $X'$
and $\pi_Y$ of $\image{f}{X}$ into $\image{f'}{X'}$, with $f'
\composition \pi_X = \pi_Y \composition f$.

Some time ago, \Jayne-\Rogers showed that a function between \Polish 
spaces is $\sigma$-continuous with closed witnesses if and only if
preimages of closed sets are \Fsigma (see \cite[Theorem 1]{JayneRogers}).
\Solecki later refined this result by providing a two-element basis,
under topological embeddability, for the family of \Baire-class one
functions which do not have this property (see \cite[Theorem 3.1]{Solecki}).
Here we use variants of the $\Gzero$ dichotomy (see \cite
{KechrisSoleckiTodorcevic}) to establish a pair of dichotomies
which together refine \Solecki's theorem.

In \S\ref{Baireclassone}, we provide a simple characterization of
\Baire-class one functions that is used throughout the remainder of the
paper. As a first application, we use the \Lecomte-\Zeleny $\Deltaclass
[0][2]$-measurable analog of the $\Gzero$ dichotomy theorem (see
\cite[Corollary 4.5]{LecomteZeleny}) to establish that the property of being
\Baire class one is determined by behaviour on countable sets.

In \S\ref{Fsigmatoone}, we use the \Hurewicz dichotomy theorem to
provide a one-element basis, under topological embeddability, for the
family of \Baire-class one functions which are not \Fsigma-to-one. To
be precise, let $\weakextendedBairespace$ denote the set
$\extendedBairespace$ equipped with the smallest topology making
the sets $\extendedextensions{s} = \set{t \in \extendedBairespace}[s
\extendedby t]$ clopen for all $s \in \Bairetree$, and fix a function $f_0 \from
\weakextendedBairespace \to \R$ such that $\restriction{f_0}{\Bairespace}$
has constant value zero and $\restriction{f_0}{\Bairetree}$ is an injection into
$\set{1 / n}[n \in \N]$.

\begin{introtheorem} \label{first}
  Suppose that $X$ and $Y$ are \Polish spaces and $f \from X \to Y$ is
  a \Baire-class one function. Then exactly one of the following holds:
  \begin{enumerate}
    \item The function $f$ is \Fsigma-to-one.
    \item There is a topological embedding of $\fzero$ into $f$.
  \end{enumerate}
\end{introtheorem}

In \S\ref{sigmacontinuous}, we use the sequential $\aleph_0$-dimensional
analog of the $\Gzero$ dichotomy theorem (i.e., the straightforward
common generalization of \cite[Theorems 18 and 21]{Miller:Survey}) to
provide a one-element basis, under topological embeddability, for the family
of \Fsigma-to-one \Baire-class one functions which are not
$\sigma$-continuous with closed witnesses. To be precise, let 
$\strongextendedBairespace$ denote the set $\extendedBairespace$
equipped with the smallest topology making the sets
$\extendedextensions{s}$ and $\set{s}$ clopen for all $s \in \Bairetree$,
and define $\fone \from \weakextendedBairespace \to
\strongextendedBairespace$ by $\fone(s) = s$.

\begin{introtheorem} \label{second}
  Suppose that $X$ and $Y$ are \Polish spaces and $f \from X \to Y$ is
  an \Fsigma-to-one \Baire-class one function. Then exactly one of the
  following holds:
  \begin{enumerate}
    \item The function $f$ is $\sigma$-continuous with closed witnesses.
    \item There is a topological embedding of $\fone$ into $f$.
  \end{enumerate}
\end{introtheorem}

Theorem \ref{second} trivially yields the following.

\begin{introtheorem}[\Jayne-\Rogers]
  Suppose that $X$ and $Y$ are \Polish spaces, and $f \from X \to Y$ is a
  function with the property that $\preimage{f}{C}$ is \Fsigma, for all closed
  subsets $C$ of $Y$. Then $f$ is $\sigma$-continuous with closed witnesses.
\end{introtheorem}

And Theorems \ref{first} and \ref{second} trivially yield the following.

\begin{introtheorem}[\Solecki]
  Suppose that $X$ and $Y$ are \Polish spaces and $f$ is a \Baire-class one
  function. Then exactly one of the following holds:
  \begin{enumerate}
    \item The function $f$ is $\sigma$-continuous with closed witnesses.
    \item There is a topological embedding of $\fzero$ or $\fone$ into $f$.
  \end{enumerate}
\end{introtheorem}

We close the paper with an appendix in which we prove several auxiliary
facts needed for the proof of the main results above. Along the way, we use
\Lecomte's $\aleph_0$-dimensional analog of the $\Gzero$ dichotomy
theorem (see \cite[Theorem 1.6]{Lecomte:Hypergraph} or \cite[Theorem 18]
{Miller:Hypergraph}) to give a new proof of a special case of \Hurewicz's
dichotomy theorem (see, for example, \cite[Theorem 21.18]{Kechris}),
yielding the existence of a one-element basis, under topological
embeddability, for the family of \Borel sets which are not \Fsigma. To be
precise, we show that if $X$ is a \Polish space and $B \subseteq
X$ is \Borel, then either $B$ is \Fsigma, or there is a topological
embedding $\pi \from \weakextendedBairespace \to X$ of $\Bairespace$
into $B$. We note that the same argument, using the parametrized
$\aleph_0$-dimensional analog of the $\Gzero$ dichotomy theorem
(i.e., the straightforward common generalization of \cite[Theorems 18
and 31]{Miller:Survey}) in lieu of its non-parametrized counterpart, yields
a slight weakening of \SaintRaymond's parametrized analog of
\Hurewicz's result (see, for example, \cite[Theorem 35.45]{Kechris}). As
a corollary, we show that \Fsigma-to-one \Borel functions between \Polish
spaces are strongly $\sigma$-closed-to-one.

\section{Baire-class one functions} \label{Baireclassone}

Throughout the rest of the paper, we will rely on the following
characterization of \Baire-class one functions.

\begin{proposition} \label{Baireclassone:covers}
  Suppose that $X$ is a topological space, $Y$ is a second
  countable metric space, and $f \from X \to Y$ is a function. Then
  the following are equivalent:
  \begin{enumerate}
    \item The function $f$ is \Baire class one.
    \item For all $\epsilon > 0$, there is a cover of $X$ by
      countably-many closed subsets whose $f$-images have
      $d_Y$-diameter strictly less than $\epsilon$.
  \end{enumerate}
\end{proposition}

\begin{proof}
  To see $(1) \implies (2)$, it is sufficient to show that for all real
  numbers $\epsilon > 0$ and open sets $V \subseteq Y$ of
  $d_Y$-diameter strictly less than $\epsilon$, the set $\preimage
  {f}{V}$ is a union of countably-many closed subsets of $X$. But
  this follows from the fact that $\preimage{f}{V}$ is \Fsigma.
  
  To see $(2) \implies (1)$, it is sufficient to show that for all real
  numbers $\epsilon > 0$ and open sets $V \subseteq Y$, there
  is an \Fsigma set $F \subseteq X$ such that $\preimage{f}
  {V_\epsilon} \subseteq F \subseteq \preimage{f}{V}$, where
  $V_\epsilon = \set{y \in Y}[\ball{d_Y}{y}{\epsilon} \subseteq V]$.
  Towards this end, fix a cover $\sequence{C_n}[n \in \N]$ of $X$
  by closed sets whose $f$-images have $d_Y$-diameter strictly
  less than $\epsilon$, define $N = \set{n \in \N}[\image{f}{C_n}
  \intersection V_\epsilon \neq \emptyset]$, and observe that the
  set $F = \union[n \in N][C_n]$ is as desired. 
\end{proof}

As a corollary, we obtain the following.

\begin{theorem} \label{Baireclassone:dichotomy}
  Suppose that $X$ and $Y$ are \Polish spaces, $d_Y$ is a
  compatible metric on $Y$, and $f \from X \to Y$ is \Borel.
  Suppose further that for all countable sets $C \subseteq X$ and
  real numbers $\epsilon > 0$, there is a \Baire-class one
  function $g \from X \to Y$ with $\sup_{x \in C} d_Y(f(x), g(x))
  \le \epsilon$. Then $f$ is \Baire class one.
\end{theorem}

\begin{proof}
  Suppose, towards a contradiction, that $f$ is not \Baire class
  one, fix a compatible metric $d_Y$ on $Y$, and appeal to
  Proposition \ref{Baireclassone:covers} to find $\delta > 0$ for
  which there is no cover of $X$ by countably-many closed subsets
  whose $f$-images have $d_Y$-diameter at most $\delta$.
  
  A \definedterm{digraph} on a set $X$ is an irreflexive subset of
  $X \times X$. A \definedterm{homomorphism} from a digraph $G$
  on $X$ to a digraph $H$ on $Y$ is a function $\phi \from X \to Y$
  sending $G$-related points to $H$-related points.

  Let $G_{\delta, f}$ denote the digraph on $X$ consisting of all
  $\pair{w}{x} \in X \times X$ for which $d_Y(f(w), f(x)) > \delta$.
  We say that a set $W \subseteq X$ is \definedterm
  {$G_{\delta, f}$-independent} if $\restriction{G_{\delta, f}}{W} =
  \emptyset$. Our choice of $\delta$ ensures that $X$ is not the
  union of countably-many closed $G_{\delta, f}$-independent sets. 
  
  Fix $s_n^{\Deltaclass[0][2]} \in \Cantortree[n]$ such that $\forall s
  \in \Cantortree \exists n \in \N \ s \extendedby s_n^{\Deltaclass[0]
  [2]}$, as well as $z_n \in \Cantorspace$ for all $n \in \N$. Now
  define a digraph on $\Cantorspace$ by setting
  \begin{equation*}
    \Gzerolittle = \set{\pair{s_n^{\Deltaclass[0][2]} \concatenation
      \sequence{0} \concatenation z_n}{s_n^{\Deltaclass[0][2]}
        \concatenation \sequence{1} \concatenation z_n}}[n \in \N].
  \end{equation*}
  The \Lecomte-\Zeleny dichotomy theorem characterizing analytic
  graphs of uncountable $\Deltaclass[0][2]$-measurable chromatic
  number (see \cite[Corollary 4.5]{LecomteZeleny}) yields a continuous
  homomorphism $\phi \from \Cantorspace \to X$ from this digraph
  to $G_{\delta, f}$. Set
  \begin{equation*}
    C = \image{\phi}{\set{s_n^{\Deltaclass[0][2]} \concatenation
      \sequence{i} \concatenation z_n}[i < 2 \mathand n \in \N]}
  \end{equation*}
  and $\epsilon = \delta / 3$.
  
  It only remains to check that no function $g \from X \to Y$ with the
  property that $\sup_{x \in C} d_Y (f(x), g(x)) \le \epsilon$ is \Baire class
  one. As $\phi$ is necessarily a homomorphism from the above
  digraph to the digraph $G_{\epsilon, g}$ associated with such a
  function, there can be no cover of $X$ by countably-many closed
  subsets whose $g$-images have $d_Y$-diameter at most
  $\epsilon$, so one more appeal to Proposition \ref
  {Baireclassone:covers} ensures that $g$ is not \Baire class one.
\end{proof}

\section{\Fsigma-to-one functions} \label{Fsigmatoone}

The proof of Theorem \ref{first} is based on a technical but useful
sufficient condition for the topological embeddability of $\fzero$.

\begin{proposition} \label{Fsigmatoone:Bairecategory}
  Suppose that $Y$ is a \Polish space and $f \from
  \weakextendedBairespace \to Y$ is a \Baire-class one function for
  which there exists $y \in Y$ such that $\Bairespace = \preimage{f}
  {y}$. Then there is a topological embedding of $\fzero$ into $f$.
\end{proposition}

\begin{proof}
  Fix a compatible metric $d_Y$ on $Y$.
  
  \begin{lemma} \label{Fsigmatoone:Bairecategory:denseopen}
    Suppose that $\epsilon > 0$. Then there is a dense open subset
    $U$ of $\weakextendedBairespace$ such that $\image{f}{U}
    \subseteq \ball{d_Y}{y}{\epsilon}$.
  \end{lemma}
  
  \begin{proof}
    By Proposition \ref{Baireclassone:covers}, there is a partition
    $\sequence{C_n}[n \in \N]$ of $\weakextendedBairespace$ into
    closed sets whose $f$-images have $d_Y$-diameter strictly less
    than $\epsilon$. Then for each non-empty open set $V \subseteq
    \weakextendedBairespace$, there exists $n \in \N$ for which
    $C_n$ is non-meager in $V$, so there is a non-empty open set
    $W \subseteq V$ such that $C_n$ is comeager in $W$. As $C_n$
    is closed, it follows that $W \subseteq C_n$, thus the diameter of
    $\image{f}{W}$ is strictly less than $\epsilon$. As $W$ necessarily
    contains a point of $\Bairespace$, it follows that $\image{f}{W}
    \subseteq \ball{d_Y}{y}{\epsilon}$. The union of the non-empty
    open sets $W \subseteq \weakextendedBairespace$ obtained in
    this way from non-empty open sets $V \subseteq
    \weakextendedBairespace$ is therefore as desired.
  \end{proof}
  
  Fix an injective enumeration $\sequence{s_n}[n \in \N]$ of
  $\Bairetree$ with the property that $s_m \extendedby s_n \implies
  m \le n$ for all $m, n \in \N$, fix a sequence $\sequence
  {\epsilon_n}[n \in \N]$ of strictly positive real numbers such that
  $0 = \lim_{n \goesto \infty} \epsilon_n$, and for each $n \in \N$, set
  $m(n+1) = \max \set{m \le n}[s_m \extendedby s_{n+1}]$. Define
  $u_\emptysequence = \emptysequence$, and recursively appeal to
  Lemma \ref{Fsigmatoone:Bairecategory:denseopen} to obtain sequences
  $u_{s_{n+1}} \in \Bairetree$, for all $n \in \N$, with the property that
  $u_{s_{m(n+1)}} \concatenation s_{n+1}(\length{s_{m(n+1)}})
  \extendedby u_{s_{n+1}}$ and $\image{f}{\extendedextensions
  {u_{s_{n+1}}}} \subseteq \ball{d_Y}{y}{\min \set{\epsilon_n, d_Y(y,
  f(u_{s_n}))}}$.
  
  Define $\pi_X \from \weakextendedBairespace \to
  \weakextendedBairespace$ by
  \begin{equation*}
    \pi_X(s) =
      \begin{cases}
        u_s & \text{if $s \in \Bairetree$, and} \\
        \union[n \in \N][u_{\restriction{s}{n}}] & \text{otherwise.}
      \end{cases}
  \end{equation*}
  Define $\pi_Y \from \image{f_0}{\weakextendedBairespace} \to \image{f}{X}$ by
  $\pi_Y(0) = y$ and $\pi_Y(f_0(s)) = f(u_s)$, for all $s \in \Bairetree$. As both of these
  functions are continuous injections with compact domains, they
  are necessarily topological embeddings, thus $\pair{\pi_X}{\pi_Y}$
  is a topological embedding of $\fzero$ into $f$.
\end{proof}

\begin{proof}[Proof of Theorem \ref{first}.]
  Proposition \ref{Fsigma:Bairespace} ensures that conditions (1) and
  (2) are mutually exclusive. To see $\neg(1) \implies (2)$, suppose that
  there exists $y \in Y$ such that $\preimage{f}{y}$ is not \Fsigma, and
  appeal to Theorem \ref{Fsigma:dichotomy} to obtain a topological
  embedding $\pi \from \weakextendedBairespace \to X$ of $\Bairespace$
  into $\preimage{f}{y}$. Proposition \ref{Fsigmatoone:Bairecategory} then
  yields a topological embedding $\pair{\pi_X}{\pi_Y}$ of $\fzero$ into $f
  \composition \pi$, and it follows that $\pair{\pi \composition \pi_X}
  {\pi_Y}$ is a topological embedding of $\fzero$ into $f$.
\end{proof}

\section{Sigma-continuous functions} \label{sigmacontinuous}

We begin with a technical but useful sufficient condition for
the topological embeddability of $\fone$.

\begin{proposition} \label{sigmacontinuous:Bairecategory}
  Suppose that $X$ and $Y$ are metric spaces, $f \from X \to Y$,
  and there are a dense \Gdelta set $C \subseteq \Bairespace$, a
  set $W \subseteq X$ intersecting the $f$-preimage of every
  singleton in a closed set, and a function $\phi \from C \to W$, such
  that both $\phi$ and $f \composition \phi$ are
  continuous, which is a homomorphism from $\restriction{\GzeroN
  [m]}{C}$ to the $\aleph_0$-dimensional dihypergraph $G_m$
  consisting of all convergent sequences $\sequence{x_n}[n \in \N]
  \in \functions{\N}{X}$ with $f(\lim_{n \goesto \infty} x_n) \neq
  \lim_{n \goesto \infty} f(x_n)$ but $\set{f(x_n)}[n \in \N] \subseteq
  \ball{d_Y}{f(\lim_{n \goesto \infty} x_n)}{1/m}$, for all $m \in \N$.
  Then there is a topological embedding of $\fone$ into $f$.
\end{proposition}

\begin{proof}
  Fix dense open sets $U_n \subseteq \Bairespace$ such that
  $\intersection[n \in \N][U_n] \subseteq C$. We will recursively
  construct sequences $\sequence{u_s}[{s \in \Bairetree[n]}]$ of
  elements of $\Bairetree$, sequences $\sequence{V_s}[{s \in
  \Bairetree[n]}]$ of open subsets of $Y$, and sequences
  $\sequence{x_s}[{s \in \Bairetree[n]}]$ of elements of $X$, for
  all $n \in \N$, such that:
  \begin{enumerate}
    \item $\forall i \in \N \forall s \in \Bairetree \ u_s
      \strictlyextendedby u_{s \concatenation \sequence{i}}$.
    \item $\forall i \in \N \forall s \in \Bairetree \ \extensions{u_{s
      \concatenation \sequence{i}}} \subseteq U_{\length{s}}$.
    \item $\forall s \in \Bairetree \ \image{(f \composition \phi)}
      {\extensions{u_s}} \union \set{f(x_s)} \subseteq V_s$.
    \item $\forall i \in \N \forall s \in \Bairetree \ V_{s
      \concatenation \sequence{i}} \subseteq V_s$.
    \item $\forall i \in \N \forall s \in \Bairetree \ \diameter{d_X}
      {\image{\phi}{\extensions{u_{s \concatenation \sequence{i}}}}}
        < 1 / \length{s}$.
    \item $\forall i \in \N \forall s \in \Bairetree \ \diameter{d_Y}
      {V_{s \concatenation \sequence{i}}} < 1 / \length{s}$.
    \item $\forall i \in \N \forall s \in \Bairetree \ \closure{\image
      {\phi}{\extensions{u_{s \concatenation \sequence{i}}}}}
        \subseteq \ball{d_X}{x_s}{1 / i}$.
    \item $\forall s \in \Bairetree \ f(x_s) \notin \closure{\union[i \in
      \N][V_{s \concatenation \sequence{i}}]}$.
    \item $\forall i \in \N \forall s \in \Bairetree \ V_{s \concatenation
      \sequence{i}} \intersection \closure{\union[j \in \N \setminus
        \set{i}][V_{s \concatenation \sequence{j}}]} = \emptyset$.
  \end{enumerate}
  We begin by setting $u_\emptysequence = \emptysequence$
  and $V_\emptysequence = Y$. Suppose now that $n \in \N$
  and we have already found $\sequence{u_s}[{s \in \Bairetree
  [\le n]}]$, $\sequence{V_s}[{s \in \Bairetree[\le n]}]$, and
  $\sequence{x_s}[{s \in \Bairetree[< n]}]$. For each $s \in
  \Bairetree[n]$, fix $\delta_s > 0$ as well as $u_s' \in \Bairetree$
  such that $u_s \extendedby u_s'$, $\extensions{u_s'} \subseteq
  U_n$, $\diameter{d_X}{\image{\phi}{\extensions{u_s'}}} < 1 / 
  n$, $\diameter{d_Y}{\image{(f \composition \phi)}{\extensions
  {u_s'}}} < 3 / 2n$, and $\ball{d_Y}{\image{(f \composition \phi)}
  {\extensions{u_s'}}}{\delta_s}\subseteq V_s$. Fix a natural
  number $n_s \ge 1 / \delta_s$ such that $u_s' \extendedby
  s_{n_s}^\N$, appeal to the \Baire category theorem
  to find $z_s \in \Bairespace$ with the property that $s_{n_s}^\N
  \concatenation \sequence{i} \concatenation z_s \in C$ for all $i
  \in \N$, and define $x_{i, s} = \phi(s_{n_s}^\N \concatenation
  \sequence{i} \concatenation z_s)$ and $y_{i, s} = f(x_{i, s})$ for
  all $i \in \N$, as well as $x_s = \lim_{i \goesto \infty} x_{i, s}$.
  The fact that $f(x_s) \neq \lim_{i \goesto \infty} y_{i, s}$ ensures
  the existence of an infinite set $I_s \subseteq \N$ for which
  $f(x_s) \notin \closure{\set{y_{i, s}}[i \in I_s]}$. Note that there
  can be no infinite set $J \subseteq I_s$ such that $\sequence
  {y_{j, s}}[j \in J]$ is constant, since otherwise the fact that
  $\image{\phi}{C} \subseteq W$ would imply that $f(x_s) =
  y_{j, s}$, for all $j \in J$. So by passing to an infinite subset
  of $I_s$, we can assume that $\sequence{y_{i, s}}[i \in I_s]$ is
  injective. By passing to a further infinite subset of $I_s$, we can
  ensure that $\sequence{y_{i, s}}[i \in I_s]$ has at most one limit
  point. By eliminating this limit point from the sequence if necessary,
  we can therefore ensure that $y_{i, s} \notin \closure{\set{y_{j, s}}
  [j \in I_s \setminus \set{i}]}$, for all $i \in I_s$. Similarly, we can
  assume that $x_s \notin \set{x_{i, s}}[i \in I_s]$. By passing one
  last time to an infinite subset of $I_s$, we can assume that $d_X
  (x_s, x_{i_{k, s}, s}) < 1 / k$ for all $k \in \N$, where $\sequence
  {i_{k, s}}[k \in \N]$ is the strictly increasing enumeration of $I_s$.
  For each $k \in \N$, fix $\epsilon_{k, s}^X > 0$ strictly less than
  $1 / k - d_X(x_s, x_{i_{k, s}, s})$, and fix $\epsilon_{k, s}^Y > 0$
  strictly less than $d_Y(f(x_s), y_{i_{k, s}, s}) / 2$ and $d_Y(y_{i, s},
  y_{i_{k, s}, s}) / 3$, for all $i \in I_s \setminus \set{i_{k, s}}$. Set
  $V_{s \concatenation \sequence{k}} = \ball{d_Y}{y_{i_{k, s}, s}}
  {\epsilon_{k, s}^Y} \intersection V_s$, and fix an initial segment
  $u_{s \concatenation \sequence{k}}$ of $s_{n_s}^\N \concatenation 
  \sequence{i_{k, s}} \concatenation z_s$ of length at least $n_s + 1$
  with the property that $\image{\phi}{\extensions{u_{s \concatenation
  \sequence{k}}}} \subseteq \ball{d_X}{x_{i_{k, s}, s}}
  {\epsilon_{k,s}^X}$ and $\image{(f \composition \phi)}{\extensions
  {u_{s \concatenation \sequence{k}}}} \subseteq V_{s \concatenation
  \sequence{k}}$. Our choice of $u_s'$ ensures that conditions (1), (2),
  and (5) hold, and along with the fact that $\phi$ is a homomorphism,
  that condition (3) holds as well. Condition (4) holds trivially, and the
  remaining conditions follow from our upper bounds on
  $\epsilon_{k, s}^X$ and $\epsilon_{k, s}^Y$. This completes the
  recursive construction.
    
  By condition (1), we obtain a continuous function $\psi \from
  \Bairespace \to \Bairespace$ by setting $\psi(s) = \union[n \in \N]
  [u_{\restriction{s}{n}}]$. Condition (2) ensures that $\image{\psi}
  {\Bairespace} \subseteq C$. Set $x_s = (\phi \composition \psi)(s)$
  for $s \in \Bairespace$, and define $\pi_X \from
  \weakextendedBairespace \to X$ and $\pi_Y \from
  \strongextendedBairespace \to Y$ by $\pi_X(s) = x_s$ and
  $\pi_Y = f \composition \pi_X$. We will show that $\pair{\pi_X}
  {\pi_Y}$ is a topological embedding of $\fone$ into $f$.
    
  \begin{lemma} \label{sigmacontinuous:Xcontainment}
    Suppose that $s \in \Bairetree$. Then $\image{\pi_X}
    {\extendedextensions{s}} \subseteq \closure{\image{\phi}
    {\extensions{u_s}}}$.
  \end{lemma}
    
  \begin{proof}
    Simply observe that
    \begin{align*}
      \image{\pi_X}{\extendedextensions{s}}
        & = \image{(\phi \composition \psi)}{\extensions{s}}
          \union \set{x_t}[t \in \extendedextensions{s} \setminus
            \extensions{s}] \\
        & \textstyle \subseteq  \image{\phi}{\extensions{u_s}}
          \union \union[t \in \extendedextensions{s} \setminus
            \extensions{s}][\closure{\image{\phi}{\extensions{u_t}}}] \\
        & \subseteq \closure{\image{\phi}{\extensions{u_s}}},
    \end{align*}
    by conditions (1) and (7).
  \end{proof}
    
  \begin{lemma} \label{sigmacontinuous:Ycontainment}
    Suppose that $s \in \Bairetree$. Then $\image{\pi_Y}
    {\extendedextensions{s}} \subseteq V_s$.
  \end{lemma}
    
  \begin{proof}
    Simply observe that
    \begin{align*}
      \image{\pi_Y}{\extendedextensions{s}}
        & = \image{(f \composition \phi \composition \psi)}
          {\extensions{s}} \union \set{f(x_t)}[t \in \extendedextensions
            {s} \setminus \extensions{s}] \\
        & \subseteq \image{(f \composition \phi)}{\extensions{u_s}}
          \union \set{f(x_t)}[t \in \extendedextensions{s} \setminus
            \extensions{s}] \\
        & \textstyle \subseteq V_s \union \union[t \in
              \extendedextensions{s} \setminus \extensions{s}][V_t] \\
        & \subseteq V_s,
    \end{align*}
    by conditions (3) and (4).
  \end{proof}
    
  To see that $\pi_X$ and $\pi_Y$ are injective, it is enough to
  check that the latter is injective. Towards this end, suppose that
  $s, t \in \extendedBairespace$ are distinct. If there is a least $n
  \le \min \set{\length{s}, \length{t}}$ with $\restriction{s}{n} \neq
  \restriction{t}{n}$, then condition (9) ensures that $V_{\restriction
  {s}{n}}$ and $V_{\restriction{t}{n}}$ are disjoint, and since Lemma 
  \ref{sigmacontinuous:Ycontainment} implies that $\pi_Y(s)$ is
  in the former and $\pi_Y(t)$ is in the latter, it follows that they are
  distinct. Otherwise, after reversing the roles of $s$ and $t$ if
  necessary, we can assume that there exists $n < \length{t}$ for
  which $s = \restriction{t}{n}$. But then condition (8) ensures that
  $\pi_Y(s) \notin V_{\restriction{t}{(n+1)}}$, while Lemma \ref
  {sigmacontinuous:Ycontainment} implies that $\pi_Y(t) \in
  V_{\restriction{t}{(n+1)}}$, thus $\pi_Y(s) \neq \pi_Y(t)$.
        
  To see that $\pi_X$ is a topological embedding, it only remains
  to show that it is continuous (since $\weakextendedBairespace$
  is compact). And for this, it is enough to check that for all $n \in
  \N$ and $s \in \weakextendedBairespace$, there is an open
  neighborhood of $s$ whose image under $\pi_X$ is a subset of 
  $\ball{d_X}{\pi_X(s)}{1 / n}$. Towards this end, observe that if $s \in
  \Bairespace$, then Lemma \ref{sigmacontinuous:Xcontainment}
  ensures that $\image{\pi_X}{\extendedextensions{\restriction{s}
  {(n+1)}}} \subseteq \closure{\image{\phi}{\extensions
  {u_{\restriction{s}{(n+1)}}}}}$, so condition (5) implies that
  $\extendedextensions{\restriction{s}{(n+1)}}$ is an open
  neighborhood of $s$ whose image under $\pi_X$ is a
  subset of $\ball{d_X}{\pi_X(s)}{1 / n}$. And if $s \in \Bairetree$,
  then Lemma \ref{sigmacontinuous:Xcontainment} ensures that
  \begin{align*}
    \textstyle \image{\pi_X}{\extendedextensions{s} \setminus
      \union[i < n][\extendedextensions{s \concatenation
        \sequence{i}}]}
      & = \textstyle \image{\pi_X}{\set{s} \union \union[i \ge n]
        [\extendedextensions{s \concatenation \sequence{i}}]} \\
      & \textstyle \subseteq \set{\pi_X(s)} \union \union[i \ge n]
        [\closure{\image{\phi}{\extensions{u_{s \concatenation
          \sequence{i}}}}}],
  \end{align*}
  so condition (7) implies that $\extendedextensions{s} \setminus
  \union[i < n][\extendedextensions{s \concatenation \sequence
  {i}}]$ is an open neighborhood of $s$ whose image under
  $\pi_X$ is a subset of $\ball{d_X}{\pi_X(s)}{1 / n}$.
    
  To see that $\pi_Y$ is continuous, it is sufficient to check that
  for all $n \in \N$ and $s \in \strongextendedBairespace$, there
  is an open neighborhood of $s$ whose image under $\pi_Y$ is
  contained in $\ball{d_Y}{\pi_Y(s)}{1 / n}$. Towards this end,
  observe that if $s \in \Bairespace$, then Lemma \ref
  {sigmacontinuous:Ycontainment} ensures that $\image{\pi_Y}
  {\extendedextensions{\restriction{s}{(n+1)}}} \subseteq
  V_{\restriction{s}{(n+1)}}$, so condition (6) implies that
  $\extendedextensions{\restriction{s}{(n+1)}}$ is an open
  neighborhood of $s$ whose image under $\pi_Y$ is
  contained in $\ball{d_Y}{\pi_Y(s)}{1 / n}$. And if $s \in \Bairetree$,
  then $\set{s}$ is an open neighborhood of $s$ whose image
  under $\pi_Y$ is a subset of $\ball{d_Y}{\pi_Y(s)}{1 / n}$.

  Before showing that $\pi_Y$ is a topological embedding, we first
  establish several lemmas.
            
  \begin{lemma} \label{sigmacontinuous:first}
    Suppose that $s \in \Bairetree$. Then $\image{\pi_Y}
    {\extendedextensions{s}} = \closure{V_s} \intersection \image
    {\pi_Y}{\strongextendedBairespace}$.
  \end{lemma}
    
  \begin{proof}
    Lemma \ref{sigmacontinuous:Ycontainment} ensures that
    $\image{\pi_Y}{\extendedextensions{s}} \subseteq \closure{V_s}
    \intersection \image{\pi_Y}{\strongextendedBairespace}$, so
    it is enough to show that $\image{\pi_Y}
    {\strongextendedBairespace \setminus \extendedextensions{s}}
    \intersection \closure{V_s} = \emptyset$. Towards this end, note
    that if $t \in \strongextendedBairespace \setminus
    \extendedextensions{s}$, then either there exists a least $n \le
    \min \set{\length{s}, \length{t}}$ for which $\restriction{s}{n}$ and
    $\restriction{t}{n}$ are incompatible, or $t \strictlyextendedby s$.
    In the former case, condition (9) implies that $\closure
    {V_{\restriction{s}{n}}}$ and $V_{\restriction{t}{n}}$ are disjoint,
    and since Lemma \ref{sigmacontinuous:Ycontainment} implies
    that $\pi_Y(t)$ is in the latter, it is not in the former. But then it is
    also not in $\closure{V_s}$, by condition (4). In the latter case,
    set $n = \length{t}$, and appeal to condition (8) to see that
    $\pi_Y(t)$ is not in $\closure{V_{\restriction{s}{(n+1)}}}$. But then
    it is also not in $\closure{V_s}$, by condition (4).
  \end{proof}
    
  \begin{lemma} \label{sigmacontinuous:second}
    Suppose that $s \in \Bairetree$. Then
    \begin{equation*}
      \textstyle
      \image{\pi_Y}{\extendedextensions{s}}
        = \image{\pi_Y}{\strongextendedBairespace} \setminus
          (\closure{\union[t \incompatible s][V_t]} \union \set{\pi_Y(t)}
            [t \strictlyextendedby s]).
    \end{equation*}
  \end{lemma}
    
  \begin{proof}
    To see that $\image{\pi_Y}{\extendedextensions{s}} \intersection
    (\closure{\union[t \incompatible s][V_t]} \union \set{\pi_Y(t)}[t
    \strictlyextendedby s]) = \emptyset$, note that if $t \incompatible s$,
    then there is a maximal $n < \min \set{\length{s}, \length{t}}$ with
    the property that $\restriction{s}{n} = \restriction{t}{n}$, in which
    case $t$ is an extension of $(\restriction{s}{n}) \concatenation
    \sequence{j}$, for some $j \in \N \setminus \set{s(n)}$. Condition
    (4) therefore ensures that
    \begin{equation*}
      \textstyle
      \closure{\union[t \incompatible s][V_t]} = \union[n < \length{s}]
        [\closure{\union[j \in \N \setminus \set{s(n)}][V_{(\restriction{s}
          {n}) \concatenation \sequence{j}}]}].
    \end{equation*}
    As Lemma \ref{sigmacontinuous:Ycontainment} implies that
    $\image{\pi_Y}{\extendedextensions{s}} \subseteq V_s$, and
    condition (4) ensures that $V_s \subseteq V_{\restriction{s}
    {(n+1)}}$ for all $n < \length{s}$, it follows from condition (9) that
    $\image{\pi_Y}{\extendedextensions{s}} \intersection \closure
    {\union[t \incompatible s][V_t]} = \emptyset$.
      
    To see that $\image{\pi_Y}{\strongextendedBairespace \setminus
    \extendedextensions{s}} \subseteq \closure{\union[t \incompatible
    s][V_t]} \union \set{\pi_Y(t)}[t \strictlyextendedby s]$, note that if
    $t \in \strongextendedBairespace \setminus \extendedextensions
    {s}$ and $t \not\strictlyextendedby s$, then there exists $n \le \min
    \set{\length{s}, \length{t}}$ such that $\restriction{s}{n}$ and
    $\restriction{t}{n}$ are incompatible, so Lemma \ref
    {sigmacontinuous:Ycontainment} ensures that $\pi_Y(t) \in \closure
    {\union[t \incompatible s][V_t]}$.
  \end{proof}
    
  \begin{lemma} \label{sigmacontinuous:third}
    Suppose that $s \in \Bairetree$. Then $\pi_Y(s)$ is the unique
    element of $\image{\pi_Y}{\strongextendedBairespace} \setminus
    (\closure{\union[t \not \extendedby s][V_t]} \union \set{\pi_Y(t)}[t
    \strictlyextendedby s])$.
  \end{lemma}
    
  \begin{proof}
    As $\closure{\union[t \not\extendedby s][V_t]} = \closure{\union[t
    \incompatible s][V_t]} \union \closure{\union[i \in \N][V_{s 
    \concatenation \sequence{i}}]}$ by condition (4), Lemma \ref
    {sigmacontinuous:second} ensures that we need only show that
    $\pi_Y(s)$ is the unique element of $\image{\pi_Y}
    {\extendedextensions{s}} \setminus \closure{\union[i \in \N][V_{s
    \concatenation\sequence{i}}]}$. Condition (8) ensures that $\pi_Y
    (s)$ is in this set, while Lemma \ref{sigmacontinuous:Ycontainment}
    implies that the other points of $\image{\pi_Y}{\extendedextensions
    {s}}$ are not.
  \end{proof}
    
  It remains to show that $\image{\pi_Y}{\extendedextensions{s}}$ and $\set
  {\pi_Y(s)}$ are clopen in $\image{\pi_Y}{\strongextendedBairespace}$,
  for all $s \in \Bairespace$. The former is a consequence of Lemmas
  \ref{sigmacontinuous:first} and \ref{sigmacontinuous:second}, while the
  latter follows from Lemma \ref{sigmacontinuous:third}.
\end{proof}

\begin{proposition} \label{closure}
  Suppose that $X$ and $Y$ are \Polish spaces, $f \from X \to Y$, $G$ is
  the $\aleph_0$-dimensional dihypergraph on $X$ consisting of all
  convergent sequences $\sequence{x_n}[n \in \N]$ such that $f(\lim_{n
  \goesto \infty} x_n) \neq \lim_{n \goesto \infty} f(x_n)$, and $W \subseteq
  X$ is $G$-independent. Then $\closure{W}$ is $G$-independent.
\end{proposition}
  
\begin{proof}
  Fix compatible metrics $d_X$ and $d_Y$ on $X$ and $Y$, respectively.
  We must show that if $x = \lim_{n \goesto \infty} \overline{w}_n$
  and each $\overline{w}_n$ is in $\closure{W}$, then $f(x) = \lim_{n
  \goesto \infty} f(\overline{w}_n)$. For each $n \in \N$, write $\overline
  {w}_n = \lim_{m \goesto \infty} w_{m, n}$, where each $w_{m, n}$ is
  in $W$. The fact that $W$ is $G$-independent then ensures that
  $f(\overline{w}_n) = \lim_{m \goesto \infty} f(w_{m, n})$. Fix $m_n \in
  \N$ with the property that both $d_X(w_{m_n, n}, \overline{w}_n)$ and $d_Y
  (f(w_{m_n, n}), f(\overline{w}_n))$ are at most $1 / n$. It follows
  that $x = \lim_{n \goesto \infty} w_{m_n, n}$, so one more appeal to
  the fact that $W$ is $G$-independent yields that $f(x) = \lim_{n
  \goesto \infty} f(w_{m_n, n}) = \lim_{n \goesto \infty} f(\overline{w}_n)$.
\end{proof}
  
\begin{proof}[Proof of Theorem \ref{second}.]
  Clearly, if (1) holds then $\preimage{f}{C}$ is \Fsigma for every closed
  set $C \subseteq Y$. Hence conditions (1) and (2) are mutually exclusive.
  To see that at least one of them holds, let $G$ denote the
  $\aleph_0$-dimensional dihypergraph on $X$ consisting of all convergent
  sequences $\sequence{x_n}[n \in \N]$ such that $f(\lim_{n \goesto \infty}
  x_n) \neq \lim_{n \goesto \infty} f(x_n)$.
  
  As $f$ is continuous on a closed set if and only if the set in question is
  $G$-independent, Proposition \ref{closure} ensures that if $X$ is a union
  of countably-many $G$-independent sets, then $f$ is $\sigma$-continuous
  with closed witnesses. We can therefore focus on the case that $X$ is not a
  union of countably-many $G$-independent sets. While it is not difficult
  to see that condition (1) fails in this case, simply applying the dichotomy
  for $\aleph_0$-dimensional analytic dihypergraphs of uncountable
  chromatic number (see \cite[Theorem 1.6]{Lecomte:Hypergraph}
  or \cite[Theorem 18]{Miller:Hypergraph}) will not yield the sort of
  homomorphism we require. So instead, we will use our further
  assumptions to obtain a homomorphism with stronger properties.
  
  Fix a compatible metric $d_Y$ on $Y$, and for each $\epsilon > 0$, let
  $G_\epsilon$ denote the $\aleph_0$-dimensional dihypergraph on $X$
  consisting of all sequences $\sequence{x_n}[n \in \N] \in G$ with $\set
  {f(x_n)}[n \in \N] \subseteq \ball{d_Y}{f(\lim_{n \goesto \infty} x_n)}
  {\epsilon}$. Note that if $B \subseteq X$ is a $G_\epsilon$-independent set
  and $C \subseteq X$ is a closed set whose $f$-image has $d_Y$-diameter
  strictly less than $\epsilon$, then $B \intersection C$ is $G$-independent.
  As Proposition \ref{Baireclassone:covers} ensures that $X$ is a
  union of countably-many closed sets whose $f$-images have
  $d_Y$-diameter strictly less than $\epsilon$, it follows that
  every $G_\epsilon$-independent set is a union of countably-many
  $G$-independent sets.
  
  We say that a set $W \subseteq X$ is \definedterm{eventually $\sequence
  {G_\epsilon}[\epsilon > 0]$-independent} if there exists $\epsilon > 0$ for
  which it is $G_\epsilon$-independent.  As $X$ is not a union of countably-many
  $G$-independent sets, it follows that it is not a union of countably-many
  eventually $\sequence{G_\epsilon}[\epsilon > 0]$-independent sets.
  Again, however, simply applying the sequential $\aleph_0$-dimensional
  analog of the $\Gzero$ dichotomy theorem (i.e., the straightforward
  common generalization of \cite[Theorems 18 and 21]{Miller:Survey})
  will not yield the sort of homomorphism we require, and we must once
  more appeal to our further assumptions.
    
  Theorem \ref{Fsigma:functiondichotomy} ensures that $X$ is a union
  of countably-many analytic sets whose intersection with the $f$-preimage
  of each singleton is closed. As $X$ is not a union of countably-many
  eventually $\sequence{G_\epsilon}[\epsilon > 0]$-independent sets, it
  follows that there is an analytic set $A \subseteq X$, whose intersection
  with the $f$-preimage of each singleton is closed, that is not a union of
  countably-many eventually $\sequence{G_\epsilon}[\epsilon > 0]$-independent
  sets.
  
  At long last, we now appeal to the sequential $\aleph_0$-dimensional analog
  of the $\Gzero$ dichotomy theorem (i.e., the straightforward common
  generalization of \cite[Theorems 18 and 21]{Miller:Survey}) to obtain a dense
  \Gdelta set $C \subseteq \Bairespace$ and a continuous function $\phi \from
  C \to A$ which is a homomorphism from $\restriction{\GzeroN[n]}{C}$ to
  $G_{1/n}$, for all $n \in \N$. In fact, by first replacing the given topology of
  $X$ with a finer \Polish topology consisting only of \Borel sets but with
  respect to which $f$ is continuous (see, for example, \cite[Theorem 13.11]
  {Kechris}), we can ensure that $f \composition \phi$ is continuous as well.
  An application of Proposition \ref{sigmacontinuous:Bairecategory} therefore
  yields the desired topological embedding of $\fone$ into $f$.
\end{proof}

\section*{Appendix: \Fsigma sets}
\renewcommand{\thesection}{\Alph{section}}
\setcounter{section}{1}
\setcounter{lemma}{0}

We begin the appendix with a straightforward observation.

\begin{proposition} \label{Fsigma:Bairespace}
  \begin{enumerate}
  \renewcommand{\labelenumi}{(\alph{enumi})}
  \item The set $\Bairespace$ is not an \Fsigma subspace of
    $\weakextendedBairespace$.
  \item The set $\Bairespace$ is a closed subspace of
    $\strongextendedBairespace$.
  \end{enumerate}
\end{proposition}

\begin{proof}
  To see (a), note that a subset of a topological space is \definedterm{\Gdelta} if it is an
  intersection of countably-many open sets. As $\Bairetree$ is
  countable and $\Bairespace$ is dense in $\weakextendedBairespace$,
  it follows that $\Bairespace$ is a dense \Gdelta subspace of
  $\weakextendedBairespace$. As $\Bairetree$ is also dense in
  $\weakextendedBairespace$, the \Baire category theorem (see, for
  example, \cite[Theorem 8.4]{Kechris}) ensures that it is not a \Gdelta
  subspace of $\weakextendedBairespace$, thus $\Bairespace$ is not
  an \Fsigma subspace of $\weakextendedBairespace$.
  
  To see (b), note that $\set{s}$ is clopen in $\strongextendedBairespace$ for all $s
  \in \Bairetree$, so $\Bairetree$ is open in
  $\strongextendedBairespace$, thus $\Bairespace$ is closed in
  $\strongextendedBairespace$.  
\end{proof}

An \definedterm{$\aleph_0$-dimensional dihypergraph} on a set
$X$ is a set of non-constant elements of $\functions{\N}{X}$. A
\definedterm{homomorphism} from an $\aleph_0$-dimensional
dihypergraph $G$ on $X$ to an $\aleph_0$-dimensional dihypergraph
$H$ on $Y$ is a function $\phi \from X \to Y$ sending elements of $G$
to elements of $H$.

Fix sequences $s_n^\N \in \functions{n}{\N}$ such that $\forall s \in
\Bairetree \exists n \in \N \ s \extendedby s_n^\N$, and define
$\aleph_0$-dimensional dihypergraphs on $\Bairespace$ by setting
\begin{equation*}
  \GzeroN[n] = \set{\sequence{s_n^\N \concatenation \sequence{i}
    \concatenation z}[i \in \N]}[z \in \Bairespace],
\end{equation*}
for all $n \in \N$, and $\GzeroN = \union[n \in \N][{\GzeroN[n]}]$.

We now establish a technical but useful sufficient condition for
the topological embeddability of $\Bairespace$.

\begin{proposition} \label{Fsigma:Bairecategory}
  Suppose that $X$ is a metric space, $Y \subseteq X$ is a set, and
  there are a dense \Gdelta set $C \subseteq \Bairespace$ and a
  continuous homomorphism $\phi \from C \to Y$ from $\restriction
  {\GzeroN}{C}$ to the $\aleph_0$-dimensional dihypergraph
  \begin{equation*}
    \textstyle G = \set{\sequence{y_n}[n \in \N] \in \functions{\N}{Y}}
      [\exists x \in X \setminus Y \ x = \lim_{n \goesto \infty} y_n].
  \end{equation*}
  Then there is a topological embedding $\pi \from
  \weakextendedBairespace \to X$ of $\Bairespace$ into $Y$.
\end{proposition}

\begin{proof}
  Fix dense open sets $U_n \subseteq \Bairespace$ such that $\intersection[n \in \N][U_n]
  \subseteq C$. We will recursively construct sequences $\sequence{u_s}[{s \in
  \Bairetree[n]}]$ of elements of $\Bairetree$ and sequences $\sequence{x_s}
  [{s \in \Bairetree[n]}]$ of elements of $X$, for all $n \in \N$, such that:
  \begin{enumerate}
    \item $\forall i \in \N \forall s \in \Bairetree \ u_s \strictlyextendedby u_{s
      \concatenation \sequence{i}}$.
    \item $\forall i \in \N \forall s \in \Bairetree \ \extensions{u_{s \concatenation
      \sequence{i}}} \subseteq U_{\length{s}}$.
    \item $\forall i \in \N \forall s \in \Bairetree \ \diameter{d_X}{\image{\phi}
      {\extensions{u_{s \concatenation \sequence{i}}}}} < 1 / \length{s}$.
    \item $\forall i \in \N \forall s \in \Bairetree \ \closure{\image{\phi}{
      \extensions{u_{s \concatenation \sequence{i}}}}} \subseteq \ball{d_X}{x_s}{1/i}$.
    \item $\forall i \in \N \forall s \in \Bairetree \ x_s \notin \closure
      {\image{\phi}{\extensions{u_{s \concatenation \sequence{i}}}}}$.
    \item $\forall i, j \in \N \forall s \in \Bairetree \ (i \neq j \implies \closure
      {\image{\phi}{\extensions{u_{s \concatenation \sequence{i}}}}} \intersection
        \closure{\image{\phi}{\extensions{u_{s \concatenation \sequence{j}}}}} =
          \emptyset)$.
  \end{enumerate}
  We begin by setting $u_\emptysequence = \emptysequence$. Suppose now that
  $n \in \N$ and we have already found $\sequence{u_s}[{s \in \Bairespace[\le n]}]$
  and $\sequence{x_s}[{s \in \Bairespace[<n]}]$. For each $s \in \Bairetree[n]$, fix
  $u_s' \in \Bairetree$ such that $u_s \extendedby u_s'$, $\extensions{u_s'} \subseteq
  U_n$, and $\diameter{d_X}{\image{\phi}{\extensions{u_s'}}} < 1 / n$, fix $n_s \in \N$
  for which $u_s' \extendedby s_{n_s}^\N$, and appeal to the \Baire category theorem
  to find $z_s \in \Bairespace$ with the property that $s_{n_s}^\N \concatenation
  \sequence{i} \concatenation z_s \in C$, for all $i \in \N$.
  Set $y_{i, s} = \phi(s_{n_s}^\N \concatenation \sequence{i} \concatenation z_s)$ for all $i
  \in \N$, as well as $x_s = \lim_{n \goesto \infty} y_{i, s}$. As $x_s \notin \set{y_{i, s}}[i \in
  \N]$, there is an infinite set $I_s \subseteq \N$ for which $\sequence{y_{i, s}}[i \in I_s]$ is
  injective. By passing to an infinite subset of $I_s$, we can assume that $d_X
  (x_s, y_{i_{k, s}, s}) < 1 / k$ for all $k \in \N$, where $\sequence{i_{k, s}}[k \in \N]$ is the
  strictly increasing enumeration of $I_s$. For each $k \in \N$, fix $\epsilon_{k, s} > 0$
  strictly less than $1 / k - d_X(x_s, y_{i_{k, s}, s})$, $d_X(x_s, y_{i_{k, s}, s})$, and $d_X
  (y_{i, s}, y_{i_{k, s}, s}) / 2$ for all $i \in I_s \setminus \set{i_{k, s}}$, and fix an initial segment
  $u_{s \concatenation \sequence{k}}$ of $s_{n_s}^\N \concatenation \sequence
  {i_{k, s}} \concatenation z_s$ of length at least $n_s + 1$ with the property that $\image
  {\phi}{\extensions{u_{s \concatenation \sequence{k}}}} \subseteq \ball{d_X}
  {y_{i_{k, s}, s}}{\epsilon_{k, s}}$. Our choice of $u_s'$ ensures that conditions (1) -- (3)
  hold, and our strict upper bounds on $\epsilon_{k, s}$ yield the remaining conditions.
  This completes the recursive construction.

  Condition (1) ensures that we obtain a function $\psi \from \Bairespace \to
  \Bairespace$ by setting $\psi(s) = \union[n \in \N][u_{\restriction{s}{n}}]$,
  and condition (2) implies that $\image{\psi}{\Bairespace} \subseteq C$.
  Set $x_s = (\phi \composition \psi)(s)$ for $s \in \Bairespace$, and
  define $\pi \from \weakextendedBairespace \to X$ by $\pi(s) = x_s$.
  We will show that $\pi$ is a topological embedding of $\Bairespace$ into $Y$.

  \begin{lemma} \label{Fsigma:containment}
    Suppose that $s \in \Bairetree$. Then $\image{\pi}{\extendedextensions{s}}
    \subseteq \closure{\image{\phi}{\extensions{u_s}}}$.
  \end{lemma}
  
  \begin{proof}
    Simply observe that
    \begin{align*}
      \image{\pi}{\extendedextensions{s}}
        & = \image{(\phi \composition \psi)}{\extensions{s}}
          \union \set{x_t}[t \in \extendedextensions{s} \setminus \extensions{s}] \\
        & \textstyle \subseteq  \image{\phi}{\extensions{u_s}} \union \union[t \in 
          \extendedextensions{s} \setminus \extensions{s}][\closure{\image{\phi}
            {\extensions{u_t}}}] \\
        & \subseteq \closure{\image{\phi}{\extensions{u_s}}},
    \end{align*}
    by conditions (1) and (4).
  \end{proof}

  To see that $\pi$ is injective, suppose that $s, t \in \extendedBairespace$
  are distinct. If there is a least $n \le \min \set{\length{s}, \length{t}}$ with
  $\restriction{s}{n} \neq \restriction{t}{n}$, then condition (6) ensures that
  $\closure{\image{\phi}{\extensions{u_{\restriction{s}{n}}}}}$ and $\closure
  {\image{\phi}{\extensions{u_{\restriction{t}{n}}}}}$ are disjoint, and since
  Lemma \ref{Fsigma:containment} implies that $\pi(s)$ is in the former and
  $\pi(t)$ is in the latter, it follows that they are distinct. Otherwise, after
  reversing the roles of $s$ and $t$ if necessary, we can assume that there
  exists $n < \length{t}$ for which $s = \restriction{t}{n}$. But then condition (5)
  ensures that $\pi(s) \notin \closure{\image{\phi}{\extensions{u_{\restriction
  {t}{(n+1)}}}}}$, while Lemma \ref{Fsigma:containment} implies that $\pi(t)
  \in \closure{\image{\phi}{\extensions{u_{\restriction{t}{(n+1)}}}}}$, thus
  $\pi(s) \neq \pi(t)$.
  
  As $\weakextendedBairespace$ is compact, it only remains to check that
  $\pi$ is continuous. And for this, it is enough to check that for all $n \in \N$
  and $s \in \weakextendedBairespace$, there is an open neighborhood of
  $s$ whose image under $\pi$ is a subset of $\ball{d_X}{\pi(s)}{1 / n}$.
  Towards this end, note first that if $s \in \Bairespace$, then Lemma \ref
  {Fsigma:containment} ensures that $\image{\pi}{\extendedextensions
  {\restriction{s}{(n+1)}}} \subseteq \closure{\image{\phi}{\extensions
  {u_{\restriction{s}{(n+1)}}}}}$, so condition (3) implies that
  $\extendedextensions{\restriction{s}{(n+1)}}$ is an open neighborhood of
  $s$ whose image under $\pi$ is a subset of $\ball{d_X}{\pi(s)}{1 / n}$.
  On the other hand, if $s \in \Bairetree$, then Lemma \ref
  {Fsigma:containment} ensures that
  \begin{align*}
    \textstyle \image{\pi}{\extendedextensions{s} \setminus \union[i < n]
      [\extendedextensions{s \concatenation \sequence{i}}]}
      & = \textstyle \image{\pi}{\set{s} \union \union[i \ge n]
        [\extendedextensions{s \concatenation \sequence{i}}]} \\
      & \textstyle \subseteq \set{\pi(s)} \union \union[i \ge n][\closure
        {\image{\phi}{\extensions{u_{s \concatenation \sequence{i}}}}}],
  \end{align*}
  so condition (4) implies that $\extendedextensions{s} \setminus \union
  [i < n][\extendedextensions{s \concatenation \sequence{i}}]$ is an open
  neighborhood of $s$ whose image under $\pi$ is a subset of $\ball{d_X}
  {\pi(s)}{1 / n}$.
\end{proof}

As a corollary, we obtain the following dichotomy theorem characterizing
the family of \Borel sets which are \Fsigma.

\begin{theorem}[\Hurewicz] \label{Fsigma:dichotomy}
  Suppose that $X$ is a \Polish space and $B \subseteq X$ is \Borel.
  Then exactly one of the following holds:
  \begin{enumerate}
    \item The set $B$ is \Fsigma.
    \item There is a topological embedding $\pi \from
      \weakextendedBairespace \to X$ of $\Bairespace$ into $B$.
  \end{enumerate}
\end{theorem}

\begin{proof}
  Proposition \ref{Fsigma:Bairespace} ensures that conditions (1) and (2) 
  are mutually exclusive. To see that at least one of them holds, let $G$
  denote the $\aleph_0$-dimensional dihypergraph consisting of all
  sequences $\sequence{y_n}[n \in \N]$ of points of $B$ converging to a 
  point of $X \setminus B$. We say that a set $W \subseteq X$ is \definedterm
  {$G$-independent} if $\restriction{G}{W} = \emptyset$. Note that the
  closure of every such subset of $B$ is contained in $B$. In particular, it
  follows that if $B$ is a union of countably-many $G$-independent sets,
  then it is \Fsigma. Otherwise, \Lecomte's dichotomy theorem for
  $\aleph_0$-dimensional dihypergraphs of uncountable chromatic number
  (see \cite[Theorem 1.6]{Lecomte:Hypergraph} or \cite[Theorem 18]
  {Miller:Hypergraph}) yields a dense \Gdelta set $C \subseteq \Bairespace$
  for which there is a continuous homomorphism $\phi \from C \to B$ from
  $\restriction{\GzeroN}{C}$ to $G$, in which case Proposition \ref
  {Fsigma:Bairecategory} yields a topological embedding $\pi \from
  \weakextendedBairespace \to X$ of $\Bairespace$ into $B$.
\end{proof}

There is also a parametrized form of this theorem.

\begin{theorem}[\SaintRaymond] \label{Fsigma:parametrizeddichotomy}
  Suppose that $X$ and $Y$ are \Polish spaces and $R \subseteq X
  \times Y$ is a \Borel set with \Fsigma horizontal sections. Then $R$
  is a union of countably-many analytic subsets with closed horizontal
  sections.
\end{theorem}

\begin{proof}
  The parametrized form of our earlier dihypergraph is given by
  \begin{equation*}
    \textstyle
    G = \set{\pair{\sequence{x_n}[n \in \N]}{y} \in \functions{\N}
      {(\horizontalsection{R}{y})} \times Y}[\exists x \in X \setminus
        \horizontalsection{R}{y} \ x = \lim_{n \goesto \infty} x_n].
  \end{equation*}
  We say that a set $S \subseteq X \times Y$ is \definedterm
  {$G$-independent} if $\horizontalsection{S}{y}$ is $\horizontalsection
  {G}{y}$-independent, for all $y \in Y$. Note that the closure of every
  horizontal section of every such subset of $R$ is contained in the
  corresponding horizontal section of $R$. Moreover, if $S \subseteq R$
  is analytic, then so too is the set $\set{\pair{x}{y} \in X \times Y}
  [x \in \closure{\horizontalsection{S}{y}}]$. In particular, it follows that if
  $R$ is a union of countably-many $G$-independent analytic subsets,
  then it is a union of countably-many analytic subsets with closed
  horizontal sections. Otherwise, the parametrized form of the dichotomy
  theorem for $\aleph_0$-dimensional dihypergraphs of uncountable \Borel
  chromatic number (i.e., the straightforward common generalization of
  \cite[Theorems 18 and 31]{Miller:Survey}) yields a dense \Gdelta set
  $C \subseteq \Bairespace$ and $y \in Y$ for which there is a
  continuous homomorphism $\phi \from C \to \horizontalsection{R}{y}$
  from $\restriction{\GzeroN}{C}$ to $\horizontalsection{G}{y}$, in which
  case Proposition \ref{Fsigma:Bairecategory} yields a continuous
  embedding $\pi \from \weakextendedBairespace \to X$ of
  $\Bairespace$ into $\horizontalsection{R}{y}$. As the latter is \Fsigma,
  this contradicts Proposition \ref{Fsigma:Bairespace}.
\end{proof}

Our use of this result is via the following corollary.

\begin{theorem} \label{Fsigma:functiondichotomy}
  Suppose that $X$ and $Y$ are \Polish spaces and $f \from X \to
  Y$ is an \Fsigma-to-one \Borel function. Then $f$ is strongly
  $\sigma$-closed-to-one.
\end{theorem}

\begin{proof}
  As the set $R = \graph{f}$ is \Borel (see, for example, \cite
  [Proposition 12.4]{Kechris}) and has \Fsigma horizontal sections,
  an application of Theorem \ref{Fsigma:parametrizeddichotomy}
  ensures that it is a union of countably-many analytic sets with
  closed horizontal sections. As the projections of these sets onto
  $X$ intersect the preimage of each singleton in a closed set, it
  follows that $f$ is strongly $\sigma$-closed-to-one.
\end{proof}

\subsection*{Acknowledgements}
We would like to thank the anonymous editor and referee for several
suggestions that improved the clarity of the exposition. The authors
were supported in part by FWF Grant P28153.

\end{document}